\documentclass[12pt]{article}

\RequirePackage{float}
\RequirePackage{enumitem}
\RequirePackage{multirow}
\RequirePackage{comment}

\RequirePackage{graphicx}
\RequirePackage{url}

\RequirePackage[dvipsnames]{xcolor}

\RequirePackage{tikz}
\usetikzlibrary{arrows.meta,positioning}

\RequirePackage{amsmath}
\RequirePackage{amsfonts}
\RequirePackage{amsthm}
\RequirePackage{amssymb}
\RequirePackage{mathtools}

\RequirePackage{soul}
\RequirePackage{framed}
\RequirePackage{subcaption}
\RequirePackage{booktabs}
\RequirePackage{arydshln}

\RequirePackage[
  colorlinks=true,        % colored text instead of boxes
  linkcolor=blue,         % internal links (sections, equations)
  citecolor=teal,         % citations
  urlcolor=blue,       % URLs
  pdfauthor={Your Name},
  pdftitle={Your Title},
  pdfsubject={Subject},
  pdfkeywords={Keyword1, Keyword2},
  pdfborder={0 0 0},      % remove link boxes if colorlinks=false
  breaklinks=true        % allow line breaks in long links (esp. for DVI->PS)
]{hyperref}
\RequirePackage[nameinlink]{cleveref} %% Load after hyperref
\RequirePackage{aliascnt}
\RequirePackage{etoolbox}

\definecolor{darkblue}{rgb}{0.0, 0.0, 0.55}
\definecolor{cobalt}{rgb}{0.0, 0.28, 0.67}

\providecommand{\demph}[1]{\emph{{\color{RoyalBlue}#1}}}

\makeatletter
\def\saveenum{\xdef\@savedenum{\the\c@enumi\relax}}
\def\resetenum{\global\c@enumi\@savedenum}
\makeatother

\newcounter{todo}

\makeatletter

\newcommand\listtodoname{List of todos}
\newcommand\listoftodos{\section*{\listtodoname}\@starttoc{tod}}
\makeatother
\numberwithin{equation}{section}

\theoremstyle{definition}
\newtheorem{theorem}{Theorem}[section]

\newaliascnt{definition}{theorem}
\newtheorem{definition}[definition]{Definition}
\aliascntresetthe{definition}
\crefname{definition}{definition}{definitions}
\Crefname{definition}{Definition}{Definitions}

\newaliascnt{lemma}{theorem}
\newtheorem{lemma}[lemma]{Lemma}
\aliascntresetthe{lemma}
\crefname{lemma}{lemma}{lemmas}
\Crefname{lemma}{Lemma}{Lemmas}

\newaliascnt{notation}{theorem}

\aliascntresetthe{notation}
\crefname{notation}{notation}{notations}
\Crefname{notation}{Notation}{Notations}

\newaliascnt{corollary}{theorem}
\newtheorem{corollary}[corollary]{Corollary}
\aliascntresetthe{corollary}
\crefname{corollary}{corollary}{corollaries}
\Crefname{corollary}{Corollary}{Corollaries}

\newaliascnt{proposition}{theorem}
\newtheorem{proposition}[proposition]{Proposition}
\aliascntresetthe{proposition}
\crefname{proposition}{proposition}{propositions}
\Crefname{proposition}{Proposition}{Propositions}

\newaliascnt{problem}{theorem}

\aliascntresetthe{problem}
\crefname{problem}{problem}{problems}
\Crefname{problem}{Problem}{Problems}

\newaliascnt{conjecture}{theorem}

\aliascntresetthe{conjecture}
\crefname{conjecture}{conjecture}{conjectures}
\Crefname{conjecture}{Conjecture}{Conjectures}

\newaliascnt{example}{theorem}
\newtheorem{example}[example]{Example}
\aliascntresetthe{example}
\crefname{example}{example}{examples}
\Crefname{example}{Example}{Examples}

\newaliascnt{remark}{theorem}
\newtheorem{remark}[remark]{Remark}
\aliascntresetthe{remark}
\crefname{remark}{remark}{remarks}
\Crefname{remark}{Remark}{Remarks}

\newaliascnt{block}{theorem}

\aliascntresetthe{block}
\crefname{block}{block}{blocks}
\Crefname{block}{Block}{Blocks}

\crefname{mainresult}{main result}{main results}
\Crefname{mainresult}{Main Result}{Main Results}

\newcommand{\exampleqedsymbol}{\hfill$\diamond$}
\AtEndEnvironment{example}{\exampleqedsymbol}

\makeatletter
\let\saved@bibitem\@bibitem
\makeatother

\makeatletter
\let\@bibitem\saved@bibitem
\makeatother

\providecommand{\bibTitleColor}[1]{{\color{RoyalBlue}#1}}

\providecommand{\bibClickableColor}{blue}
\providecommand{\bibClickableStyle}[1]{{\color{\bibClickableColor}[#1]}}

\RequirePackage{etoolbox} % for environment hooks
\AtBeginEnvironment{thebibliography}{\hypersetup{urlcolor=\bibClickableColor}}
\usepackage[
  left=1.5in,
  right=1.5in,
  top=1in,
  bottom=1in
]{geometry}

\usepackage{mathrsfs}

\def\bfw{\mathbf{w}}
\def\bfu{\mathbf{u}}
\def\bfx{\mathbf{x}}

\newcommand\pp{{\mathbb{P}}}

\newcommand\cc{{\mathbb{C}}}

\DeclareMathOperator{\GED}{gEDdeg}                  
\DeclareMathOperator{\UED}{uEDdeg}                  
\DeclareMathOperator{\DED}{EDdefect}

                \DeclareMathOperator{\codim}{codim}                              \DeclareMathOperator{\reg}{reg}                  \DeclareMathOperator{\sing}{Sing}

\DeclareMathOperator{\Eu}{Eu}

\DeclareMathOperator{\rank}{Rank}

\DeclareMathOperator{\EDdeg}{EDdeg}

\def\bC{\mathbb{C}}

\def\bP{\mathbb{P}}

\def\lra{\longrightarrow}
\def\bQ{\mathbb{Q}}
\def\bs{{\bf s}}

\def\cL{\mathcal{L}}

\def\C{\mathbb{C}}

\newcommand{\PP}{\mathbb{P}}

\newcommand{\CC}{\mathbb{C}}

\usepackage{array}

\title{Euclidean distance degree defect \\of singular projective varieties}

\author{Lauren\c{t}iu G. Maxim, Jose Israel Rodriguez, Botong Wang}

\newcommand{\XWhitney}{\mathscr{X}}
\begin{document}

\date{}
\maketitle

\begin{abstract}    
The unit Euclidean distance degree and the generic Euclidean distance degree are two well-studied invariants of projective varieties. These quantities measure the algebraic complexity of nearest-point problems on a variety, and in many examples arising in optimization, engineering, statistics, and data science, there is a significant gap between them. We refer to this difference as the defect of the Euclidean distance (ED) degree.

In this paper, we provide a constructible enhancement and a topological formula for the defect of the ED degree of an arbitrary complex projective variety, extending our previous results from the smooth setting. Since the generic Euclidean distance degree is typically more tractable, our approach offers a new method for computing ED degrees in broad generality.
\end{abstract}

\tableofcontents

\section{Introduction}

\subsection{Background on ED degrees}
The notion of Euclidean distance degree was introduced in \cite{DHOST} as a measure of the algebraic complexity for 
nearest
point problems of an algebraic variety.
It counts the number of complex critical points of the squared Euclidean distance function from a general data point to a given variety.

A natural
generalization arises when one considers weighted least squares optimization as follows.

\begin{definition}
Let $X$ be an irreducible closed subvariety of $\cc^n$. 
For a given weight $\bfw=(w_1,\dots,w_n) \in (\C^*)^n$, the
\demph{$\bfw$-weighted Euclidean distance degree} of $X$ is the 
 number of complex critical points of 
 \[
 d_{\bfu,\bfw}(\bfx):=\sum_{i=1}^n w_i(x_i-u_i)^2
 \] 
on the smooth locus $X_{\reg}$ of $X$, for general data $\bfu=(u_1,\dots,u_n) \in \C^n$.
 We denote this degree by $\EDdeg_\bfw(X)$. 
 \end{definition}

Many models in applications are invariant under rescaling 
and are therefore naturally viewed as projective varieties.
The notion of $\bfw$-weighted ED degree can be defined for these models as well, using affine cones.

\begin{definition}\label{def:projective-weighted-ed-degree}
If $X$ is an irreducible closed subvariety of the complex projective space $\pp^n$, we define the (projective) $\bfw$-weighted Euclidean distance degree of $X$ by
\[\EDdeg_\bfw(X) := \EDdeg_\bfw(C(X)),\]
where $C(X)$ is the affine cone of $X$ in $\cc^{n+1}$. 
 \end{definition}

When the weight vector $\bfw$ is general, we refer to $\EDdeg_\bfw(X)$ as the \demph{generic ED degree} and denote it by $\GED(X)$. When $\bfw=\bf{1}$, the all-ones vector, we obtain the \demph{unit ED degree}, denoted $\UED(X)$. While these notions can be defined in both affine and projective settings, in this paper we focus on the projective setting.

It was proved in \cite[Theorem 1.3]{MRWp} (see also \cite[Theorem 8.1]{AH} for the smooth case) that the unit ED degree 
of a projective variety can be computed as an Euler characteristic weighted by a certain constructible function. More precisely, one has the following result.
\begin{theorem}\label{thm1}
Let $X \subset \bP^n$ be an irreducible closed subvariety. Then
\begin{equation}\label{eq1} \UED(X)= (-1)^{\dim (X)} \chi( {\rm Eu}_X \vert_{\pp^n  \setminus (Q \cup H)}),\end{equation}
where ${\rm Eu}_X$ is MacPherson's local Euler obstruction function on $X$, 
$Q$ is the isotropic quadric 
\[Q:=\left\{[x_0:\dots:x_n]\in\pp^n\mid x_0^2+x_1^2+\cdots+x_n^2=0\right\},\]
$H$ is a general hyperplane in $\bP^n$, and $\chi(-)$ denotes the Euler characteristic of a constructible function. In particular, if $X$ is smooth, then 
\begin{equation}\label{eq2}\UED(X)= (-1)^{\dim (X)} \chi(X \setminus (Q \cup H)).\end{equation}
\end{theorem}

Moreover, as already noted in \cite{ED-degree-MR4136171}, the proof of Theorem \ref{thm1} can be extended to the computation of $\EDdeg_\bfw(X)$, for an arbitrary weight $\bfw \in (\C^*)^{n+1}$. Indeed, if one replaces $Q$ by
the quadric 
\[Q_{\bfw}=\left\{[x_0:\dots:x_n]\in\pp^n\mid w_0x_0^2+w_1x_1^2+\cdots+w_nx_n^2=0\right\},\] the following holds.

\begin{theorem}\label{thm2}
Let $X \subset \bP^n$ be an irreducible closed subvariety. Then,
\begin{equation}\label{eq3} 
\EDdeg_\bfw(X)= (-1)^{\dim (X)} \chi( {\rm Eu}_X \vert_{\pp^n \setminus (Q_\bfw \cup H)}),
\end{equation}
with $H$ a general hyperplane. 
In particular, if $X$ is smooth, one has the equality 
\begin{equation}\label{eq4}
\EDdeg_\bfw(X)= (-1)^{\dim (X)} \chi(X \setminus (Q_\bfw \cup H)).
\end{equation}
\end{theorem}

These formulas give a topological description of the weighted ED degree allowing 
techniques from topology, algebraic geometry, as well as  singularity theory, to be applied to their computation.

 The unit ED degree 
 is in general difficult to compute even if $X$ is smooth, since the isotropic quadric $Q$ may intersect $X$ non-transversally. 
On the other hand, for general weight $\bfw$, the quadric $Q_\bfw$ intersects $X$ transversally, and the computation of $\GED(X)$ is more manageable, e.g., (see \cite{DHOST}, \cite{OSS2014}).

\subsection{ED degree defect and statements of results}\label{ss:statement-results}

In this paper, we evaluate the difference between generic and unit ED degree in terms of invariants of the singularities of 
$X\cap Q$. This leads to formulas for the unit ED degree that may otherwise be difficult to obtain directly. 
Of course, if the difference is zero then the unit ED degree equals the more amenable generic ED degree. 
That is,
we analyze the difference $$\DED(X)\coloneqq\GED(X)-\UED(X),$$  which in \cite{ED-degree-MR4136171} was referred to as the \demph{defect} of the Euclidean distance degree. It is known that $\DED(X)$ is non-negative, and for many varieties appearing in optimization, engineering, statistics, and data science, the defect is  substantial.

\medskip

\subsubsection*{A pencil of quadric  sections}

Let $X$ be an irreducible complex projective variety in $\pp^n$. We assume throughout that $X$ is not contained in the isotropic quadric  
$$Q=\{[x_0:\ldots:x_n]\in \PP^n: x_0^2+\cdots+x_n^2=0\}.$$
We also assume $\bfw\in (\CC^*)^{n+1}$ is general when considering
the pencil of quadrics parameterized  by $\lambda=[\lambda_0:\lambda_1]\in \pp^1$:
\[Q_{\lambda}:=\{f_{\lambda}=
{\lambda_0} (w_0x_0^2+\cdots+w_nx_n^2)
+
\lambda_1(x_0^2+\cdots+x_n^2)
=0
\}.\]
The restriction  to $X$ gives a 
 pencil of quadric  sections
\[
\{X_{\lambda} :=X \cap Q_{\lambda}\}_{\lambda\in \PP^1}
\]
on $X$. 
A general member  of this pencil 
corresponds to ${\lambda}=[1:0]$ and is given by
\[	X_{\infty}:=
X \cap Q_\infty
\]
with $Q_\infty=Q_\bfw$.
The special member of the pencil corresponds to ${\lambda}=[0:1]$ and is given by
\[	X_0:=X \cap Q_0 
\] 
for $Q_{0}=Q$ the isotropic quadric. 
In terms of $f_\lambda$ we have 
$X_{\infty}=X \cap \{f_\infty=0\}$ and 
$X_0 =X \cap \{f_0=0\}$,
where 
\[
f_0(x)=x_0^2+\cdots+x_n^2
\quad\text{and}\quad
f_\infty(x)=w_0x_0^2+\cdots+w_nx_n^2.
\]

\subsubsection*{Incidence variety of the pencil and the associated vanishing cycles
}

We fix a Whitney stratification $\XWhitney$ of $X$. 
We first note that, for a general choice of the weight $\bfw$, the quadric $Q_\infty=Q_\bfw$ intersects all strata 
transversally.  Indeed, since  $x_0^2,\ldots,x_n^2$ have no common zero in $\pp^n$, Bertini's theorem implies that a general quadric $Q_\bfw=\{\sum_i w_i x_i^2=0\}$ meets each stratum of $X$ transversally. As there are only finitely many strata, a general choice of $\bfw$ ensures transversality on all strata at once.

Consider the incidence variety of the pencil, that is, 
\begin{equation}\label{eq:tilde-X-incidence-pencil}
\widetilde{X}:=\{({\lambda},x)\in \pp^1 \times X \mid x \in X_{\lambda}\}
\end{equation}
with projection maps 
\[\pi:\widetilde{X} \to \pp^1\quad \text{and}\quad p:\widetilde{X} \to X.\]
Note that $\widetilde{X}$ is the restriction to $\pp^1 \times X$ of the blowup of $\pp^n$ along the base locus 
$B:=Q_0 \cap Q_{\infty}$ of the pencil.
We have $X_{\lambda}\cong \widetilde{X}_{\lambda}:=\pi^{-1}({\lambda})$ for any ${\lambda} \in \pp^1$. 

We define a regular function
\begin{equation}\label{eq:our-function-f}
f:=\frac{f_0}{f_\infty}\colon X\setminus X_\infty \lra \bC,
\end{equation}
by restricting the function $f_0/f_\infty$, originally defined on $\pp^n\setminus Q_\infty$, to $X\setminus X_\infty$. 
By construction, the zero fiber of $f$ is
\[
f^{-1}(0)=X_0\setminus X_\infty.
\]

The main ingredient in our new formula for the ED degree defect involves the vanishing cycle of the function $f$, as a transformation of constructible functions. To help set the notation, we review some of the background here. For textbook references on nearby and vanishing cycles see \cite{Dim04,Max, MS22}.  For further background material on Whitney stratifications and  constructible functions we refer the reader to our survey \cite{MRWh}; see also \cite{ED-degree-MR4136171}.

Fix $\alpha\in CF(Z)$ a constructible function on a variety $Z$, and
let $\mathscr{Z}$ be a Whitney stratification of $Z$ such that $\alpha$
is constant on each stratum of $\mathscr{Z}$.
Then, the \demph{Euler characteristic of $\alpha$} 
is
 the weighted Euler characteristic defined as
\begin{equation}\label{wec}
\chi(\alpha)
:=\sum_{V \in \mathscr{Z}} \chi(V) \cdot \alpha(V),
\end{equation} 
where $\alpha(V)$ denotes the (constant) value of $\alpha$ on the stratum $V \in \mathscr{Z}$ and  $\chi(V)$ is the topological Euler characteristic of $V$.

For 
 $f:X\setminus X_\infty \to \bC$ as above,
 the nearby cycle  operator
\begin{equation}\label{eq:define-nearby-cycle-functor}
\psi_f:CF(X\setminus X_\infty) \to CF(X_0\setminus X_\infty)
\end{equation}
 is defined on the groups of constructible functions 
 $CF(X\setminus X_\infty)$ and valued on $CF(X_0\setminus X_\infty)$,  
 as follows. 
For $\alpha \in CF(X\setminus X_\infty)$ a constructible function on $X\setminus X_\infty$, $\psi_f(\alpha)$ is the constructible function on $X_0\setminus X_\infty$ whose value at $x \in X_0\setminus X_\infty$ is given by
$$\psi_f(\alpha)(x):=
\chi(\alpha|_{F_x}),$$
where $F_x$ denotes the Milnor fiber of $f$ at $x$. 

\begin{definition}
With the notation above, the \demph{vanishing cycle operator}  of $f$
is  defined as
\[
\varphi_f:CF(X\setminus X_\infty) \to CF(X_0\setminus X_\infty),\quad 
\alpha\mapsto \psi_f(\alpha) - \alpha\vert_{f^{-1}(0)}. 
\]
\end{definition}
The definition can be restated 
in terms of the quadrics $Q_\bfw$ and $Q$ as
\[
\varphi_f:CF(X\setminus  Q_\bfw) \to CF((X\cap Q)\setminus  Q_\bfw),\quad 
\alpha\mapsto \psi_f(\alpha) - \alpha\vert_{X\cap Q \setminus  Q_\bfw}.
\]
When $\alpha=\Eu_{X\setminus Q_\bfw}$ is the local Euler obstruction function of $X\setminus Q_\bfw$,  
we have a constructible function 
\begin{equation}\label{eq:Eu-X-minus-Q-w}
\varphi_f(\Eu_{X\setminus Q_\bfw}):(X\cap Q)\setminus Q_\bfw\to \mathbb{Z}.
\end{equation}

Next, 
for $\bfu\in \CC^{n+1}\setminus \{0\}$
let $H_\bfu$ denote the  hyperplane 
\[H_\bfu:=\{
[x_0:\dots:x_n]\in \PP^n: 
u_0x_0+\cdots+u_nx_n=0
\}.\]  
We denote the restriction of \eqref{eq:Eu-X-minus-Q-w} to the complement of $H_\bfu$ by
\[
\varphi_f(\Eu_{X\setminus Q_\bfw})|_{(X\cap Q)\setminus (Q_\bfw\cup H_\bfu )}. 
\]
This constructible function appears in our main result on the ED degree defect when $H_\bfu$ is a general hyperplane. 
Our result is formulated in terms of the Euler characteristic of a constructible function, as defined in \eqref{wec}.

\subsubsection*{ED degree defect as a topological invariant}

A topological interpretation of the ED degree defect in terms of invariants of singularities of $X \cap Q$ was given in \cite{ED-degree-MR4136171} in the case when $X$ is a smooth irreducible complex projective variety in $\pp^n$, by using the above pencil of quadrics on $X$ and the associated vanishing cycles. In this paper we extend the results of \cite{ED-degree-MR4136171} to an arbitrary complex projective variety $X$. 
A main motivator of this extension is because many models in applications are singular.
As already noticed in \cite{ED-degree-MR4136171}, the ED degree defect can in general be computed 
much more easily than
computing $\GED(X)$ and $\UED(X)$ individually.

With the above assumptions and notation, our first result, generalizing \cite[Theorem 3.2]{ED-degree-MR4136171}, can be formulated as follows.
\begin{theorem}\label{th-main}
Let $X \subset \pp^n$ be an irreducible complex projective variety, and let $\bfw$ be a general weight. 
Let $f$ be defined as in \eqref{eq:our-function-f}.
Then,
	 \begin{equation}\label{th-eq}
	\DED(X)
	= 	(-1)^{\dim (X)-1} \chi\big((\varphi_{f} (\Eu_{X \setminus Q_\bfw}))|_{X_0\setminus(Q_\bfw \cup H)}\big)
	 \end{equation}
     where $H$ is a general hyperplane in $\mathbb{P}^n$.
\end{theorem}

So the function $\varphi_{f} (\Eu_{X \setminus Q_\bfw})|_{X_0\setminus(Q_\bfw \cup H)}$ provides a constructible enhancement of the ED degree defect.
 A different ED defect formula, expressed in terms of Segre classes, was obtained with completely different methods in \cite[Equation (4--4)]{AH}.

We further note that for general weight $\bfw$, the constructible function $\varphi_{f} (\Eu_{X \setminus Q_\bfw})|_{X_0\setminus(Q_\bfw \cup H)}$ is in fact supported on the 
intersection of $f^{-1}(0)=X_0\setminus X_\infty$ with the stratified singular locus of $f$, i.e., on 
\[
X_0 \cap \sing_{\mathscr{X}}(f)\setminus (Q_\bfw \cup H),
\]
where $\mathscr{X}$ is a Whitney stratification of $X$ as before. Let
$$Z:=\overline{X_0 \cap \sing_{\mathscr{X}}(f)},$$
the closure being taken in $X_0$.
In particular, since $\sing_{\mathscr{X}}(f)$ is closed in $X\setminus Q_\bfw$, we get that $Z \setminus Q_\bfw=X_0 \cap \sing_{\mathscr{X}}(f)$.
Note for a general $\bfw$, $Z$ is independent of $\bfw$ and in fact can be identified with the locus where $Q$ fails to meet the strata of $\mathscr{X}$ transversally.
Refine the stratification $\XWhitney$ within $Z$ in such a way that the  function $\varphi_{f} (\Eu_{X \setminus Q_\bfw})|_{X_0\setminus(Q_\bfw \cup H)}$ is constructible with respect to the new stratification. This refinement can be chosen independently of $\bfw$.  
Denote by $\mathscr{X}_0$ the collection of the refined Whitney strata contained in~$Z$. 
From here on, we choose $\bfw$ and $H$ to be general with respect to the above refinement. 
For any stratum $V \in \mathscr{X}_0$, we denote by 
${\overline{V}}$ its closure.
Let $CF_{\mathscr{X}_0}(Z)$ denote the abelian group of $\mathscr{X}$-constructible functions on $X_0$ which are supported on $Z$. Then it is well-known that the collection $\{ {\rm Eu}_{{\overline{V}}} \mid V \in \mathscr{X}_0 \}$ is a basis of $CF_{\mathscr{X}_0}(Z)$. Expressing the constructible function \begin{equation}\label{mu} \mu:=\varphi_{f}(\Eu_{X \setminus Q_\bfw})|_{X_0\setminus(Q_\bfw \cup H)}\end{equation} of Theorem \ref{th-main} in the basis $\{ {\rm Eu}_{{\overline{V}} \setminus (Q_\bfw \cup H)} \mid V \in \mathscr{X}_0 \}$ of constructible functions with support on $Z \setminus (Q_\bfw \cup H)$
yields our second main result, which generalizes \cite[Theorem 1.5]{ED-degree-MR4136171} to the singular context:

\begin{theorem}\label{thm-i1}  
Let $X \subset \pp^n$ be an irreducible projective variety not contained in the isotropic quadric $Q$. Then, 
\begin{equation}\label{i3}
\DED(X)
=\sum_{V \in \mathscr{X}_0} (-1)^{\codim_{X_0}(V)} \alpha_V \cdot\GED({\overline{V}})
\end{equation}
with 
\begin{equation}\label{alv}
\alpha_V
    =
\mu_V-
    \sum_{
    \{S \mid V 
    \subset {\overline S} \setminus S\}
    } 
    \chi_c(L_{V,S}) \cdot \mu_S,\end{equation}
where, for any stratum $V\in \mathscr{X}_0$, $\mu_V=\mu(V)$ is the value of the constructible function $\mu$ of \eqref{mu} at a point $v\in V\setminus (Q_\bfw \cup H)$, the summation in \eqref{alv} is over strata $S$ in $Z$ different from $V$, which contain $V$ in their closure,
 and $L_{V,S}$ is the complex link\footnote{For the definition of the complex link $L_{V,S}$ of $V$ in $\overline S$ see, e.g.,  \cite[Page 15]{GM}.} of a pair of distinct strata $(V,S)$ with 
 $V \subset {\overline S}$.
 \end{theorem} 

 \begin{corollary}\label{cor28}
     If $Z$ consists only of isolated stratified singularities of $f$, then  
     \begin{equation}\label{isp}
\DED(X)
=\sum_{P \in Z} (-1)^{\dim(X)-1} \mu_P,
\end{equation}
where $\mu_P=\mu(P)=\varphi_{f}(\Eu_{X \setminus Q_\bfw})(P)$.
 \end{corollary}

As in \cite[Corollary 1.10]{ED-degree-MR4136171}, we also have the following result motivated by the duality conjecture of \cite[Equation (3.5)]{OSS2014} in structured low-rank approximation.   
\begin{corollary}[Intersection with linear space]\label{cor-slice}  
With the notation as in Theorem~\ref{thm-i1},
let $\cL$ denote a general linear subspace of $\pp^n$.
Then
\begin{equation}
\DED(X\cap\cL)
=\sum_{V \in \mathscr{X}_0} (-1)^{\codim_{X_0} V} \alpha_V \cdot\GED({\overline{V}}\cap \cL).
\end{equation}
\end{corollary}

\begin{remark}
    When $X$ is smooth, we have that
    \[\mu_V=\chi(\widetilde{H}^*(F_{V};\bQ))\]
    is the Euler characteristic of the reduced cohomology of the Milnor fiber $F_{V}$ of the hypersurface $X_0 \subset X$ at some point in $V$, e.g., it coincides (up to sign) with the Milnor number if $V$ is an isolated singular point. So in this case the ED degree defect is computed by \eqref{i3} only in terms of Milnor fiber and complex link information, thus recovering \cite[Theorem 1.5]{ED-degree-MR4136171}.  Furthermore, as noted in Section \ref{sec:comp}, the negative of the integer $\mu_V$ matches what, in \cite[Section 5]{BMPS}, was called the ``defect'' of the function $f$ on $X$ at some point of $V$.
\end{remark}

\medskip

An additional goal of this paper is to provide the tools to compute defects of singular models in applications. For instance, results in the previously mentioned paper \cite{ED-degree-MR4136171} were used to prove results on the defect of smooth rank one tensors and matrices. A first obstruction to obtain results on the defect of higher rank matrices was the omission of a topological formula for the ED degree defect of singular models. Our \Cref{th-main} and \Cref{thm-i1} fill this void.

\medskip
The remainder of our paper is organized as follows. 
Proofs of our main results  \Cref{th-main} and \Cref{thm-i1} will be given in Section \ref{proofs-sec}. 
Concrete examples are worked out in Section \ref{ex-sec}
using our main results,
and additional 
computations of defects are performed in \Cref{s:aag} by numerical methods.

\section{Proofs and discussion of main results}\label{proofs-sec}

In this section we provide
the proofs of our main results, Theorems \ref{th-main}
and \ref{thm-i1}.
We also discuss computational aspects of these results.

\subsection{Proof of the topological formula for ED degree~defect}

Let $M$ be a smooth complex projective variety, and let $L$ be a line bundle on $M$. Suppose that $s_0, \dots, s_n\in H^0(M, L)$ span a base-point-free linear system. For a general choice of $\mathbf{a}=(a_0, \dots, a_n)\in \C^{n+1}$, let $B_{\mathbf a}$ be the base locus (as a subscheme of $M$) of the pencil determined by the sections $s_0$ and $a_0s_0+\cdots+ a_ns_n$, i.e.,
\[
B_{\mathbf a}=\{s_0=a_0s_0+\cdots+ a_ns_n=0\}\subset M.
\]
Let $p: \widetilde{M}\to M$ be the blowup of $M$ along $B_{\mathbf a}$ with exceptional divisor $E$. Then we have a morphism 
\[
\pi \coloneqq [s_0:a_0s_0+\cdots+ a_ns_n]: \widetilde{M}\to \mathbb{P}^1.
\]
Consider $\infty=[1:0]\in \mathbb{P}^1$, and $\C=\mathbb{P}^1\setminus \{\infty\}$. Let $U=\pi^{-1}(\C)$, and denote by $\pi_U: U\to \C$ the restriction of $\pi$.

Next, we give a more general formulation of \cite[Proposition A.1]{LMW25}, while the proof follows the same lines of reasoning as in loc.cit.

\begin{proposition}\label{p31}
    Fix a Whitney stratification $\mathcal{W}$ of $M$. Then, for a general $\mathbf{a}=(a_0, \dots, a_n)\in \C^{n+1}$ and under the above notation, the following properties hold. 
    \begin{enumerate}
        \item 
        Denote by $\mathcal{W}'$ the pullback stratification of $\mathcal{W}$ by $p$. Then the restriction of $\mathcal{W}'$ to $U$ is a Whitney stratification.  
        \item The map $\pi_U: U\to \C$ does not have critical points along $E$ with respect to the Whitney stratification $\mathcal{W}'$.
    \end{enumerate}
\end{proposition}

\begin{remark}\label{r32}
  In this paper we work with a projective variety $X \subset \pp^n$. In order to apply Proposition \ref{p31} to the possibly singular variety $X$ it suffices to work with a Whitney stratification $\mathcal{W}$ of $M=\pp^n$ so that $X$ is a union of strata. We let $L=\mathcal{O}_{\pp^n}(2)$, and $s_0=x_0^2+\ldots + x_n^2$, together with $s_i=x_i^2$ for $i=1,\ldots,n$. We then have by part (1) of Proposition~\ref{p31} that $p^{-1}(X)=\widetilde{X}$ is a union of strata of $\mathcal{W}'$. We use the same notation $\pi$ and $\pi_U$
for their restrictions to $\widetilde{X}$ and
$U\cap\widetilde{X}$, respectively.
By part (2) of Proposition~\ref{p31}, the restricted morphism $\pi_U$
has no stratified critical points along $E\cap U\cap\widetilde{X}$.
Moreover, $p$ induces an isomorphism
\[
(U\cap\widetilde{X})\setminus E
=\widetilde{X}\setminus(E\cup\widetilde{X}_{\infty})
\xrightarrow{\;\sim\;}X\setminus X_\infty,
\]
under which the restriction of $\pi_U$ identifies with the
function $f:X\setminus X_\infty\to\C$ introduced 
earlier.
    \hfill$\diamond$
\end{remark}

In the proof of Theorem \ref{th-main}, we will need the following.
\begin{lemma}\label{baseloc}
    In the above notation,
\begin{equation}\label{eq6} \chi(\Eu_X\vert_{X_\infty}) - \chi(\Eu_X\vert_{X_0}) = \chi(\varphi_{f} (\Eu_{X \setminus X_\infty})).  
\end{equation}
\end{lemma}

\begin{proof} 
 Throughout this proof, we use $p$, $\pi$, and $\pi_U$ for their
restrictions to $\widetilde{X}$ and $U\cap\widetilde{X}$, as
appropriate, and write $U$ and $E$ for $U\cap\widetilde{X}$
and $E\cap\widetilde{X}$, respectively
(cf.\ Remark~\ref{r32}).

First note that  for general $\lambda$ near the origin of $\C$ we have 
{stratified homeomorphisms}
$X_\infty \cong X_{\lambda} \cong \widetilde{X}_{\lambda}$. Moreover, if $i_{\lambda}:X_{\lambda} \hookrightarrow X$ and $\widetilde{i}_{\lambda}:\widetilde{X}_{\lambda} \hookrightarrow \widetilde{X}$ denote the closed inclusions for
${\lambda} \in \C$, then under the isomorphism $X_{\lambda} \cong \widetilde{X}_{\lambda}$, the map $i_{\lambda}$ can be identified with the composition $p\circ \widetilde{i}_{\lambda}$. With these identifications and recalling the notation from the beginning of this section (and Remark \ref{r32}), we have
\[
\begin{split}
    \chi(\Eu_X\vert_{X_\infty}) - \chi(\Eu_X\vert_{X_0})&=\chi(p^*(\Eu_X)\vert_{\widetilde{X}_{\lambda}}) - \chi(p^*(\Eu_X)\vert_{\widetilde{X}_0})\\
    &\overset{(a)}=\chi(\psi_{\pi_U}(p^*(\Eu_X)\vert_U)) - \chi(p^*(\Eu_X)\vert_{\widetilde{X}_0}) \\
    &\overset{(b)}=\chi(\varphi_{\pi_U}(p^*(\Eu_X)\vert_U)),
\end{split}
\]
with $p^*(\Eu_X)=\Eu_X \circ p \in CF(\widetilde{X})$ denoting the pullback of the constructible function $\Eu_X$ under $p$. Here, the equality $(a)$ uses the properness of $\pi_U$ and \cite[Example 10.4.20]{MS22}, and $(b)$ follows from the definition of vanishing cycles. Next note that by Proposition \ref{p31}, the support of $\varphi_{\pi_U}(p^*(\Eu_X)\vert_U)$ does not intersect the exceptional divisor $E$, hence 
\begin{equation}\label{lemvan}
\chi(\varphi_{\pi_U}(p^*(\Eu_X)\vert_U))=\chi(\varphi_{\pi_U\vert_{U \setminus E}}(p^*(\Eu_X)\vert_{U\setminus E}))=\chi(\varphi_{f} (\Eu_{X \setminus X_\infty})),
\end{equation}
where the first equality uses the fact that vanishing cycles commute with open
embeddings (cf. also \cite[Proposition 10.4.19(2)]{MS22}), and the second equality uses the identifications of Remark \ref{r32}.
\end{proof}

\begin{remark} In the notation from the beginning of this section, one could also prove directly that, for a general choice of $\mathbf{a}=(a_0, \dots, a_n)$ and every $\mathcal{W}$-constructible function $\alpha$ on $M$, one has
\begin{equation}\label{vanvan}
\varphi_{\pi_U}(p^*\alpha\vert_U)\vert_{E \cap \pi^{-1}(0)}=0.
\end{equation}
Then in the setting of Remark \ref{r32}, applying the above formula to $\Eu_X$ extended by zero to $\pp^n$, we get:
\[
\varphi_{\pi_U}(p^*\Eu_X\vert_{U\cap \widetilde{X}})\vert_{E \cap \widetilde{X} \cap \pi^{-1}(0)}=0.
\]
So in the last part of the proof of Lemma \ref{baseloc}, the support of $\varphi_{\pi_U}(p^*\Eu_X\vert_{U})$ is disjoint from $E$, giving the first equality in \eqref{lemvan}; this is the only place in the proof where Proposition \ref{p31} was mentioned, and formula \eqref{vanvan} can be used as a substitute.

To justify formula \eqref{vanvan}, let $s_\mathbf{a}=a_0s_0+\cdots+a_ns_n$. By resolution of singularities and additivity, the claim reduces (as in part (b) of \cite[Proposition 4.1]{MSS}) to constructible functions of the form $q_* 1_Y$, with $q:Y \to M$ proper, $Y$ smooth, and the divisor $\{q^*s_0=0\}$ being simple normal crossing. For general $\mathbf{a}$, the divisor $\{q^*s_\mathbf{a}=0\}$ is smooth and transversal to the strata of $\{q^*s_0=0\}$. By proper base change and compatibility of vanishing cycles with proper pushforward, we can then reduce the desired vanishing the the normal crossing case (proved in part (a) and not requiring very ampleness)\footnote{In fact, while Proposition 4.1 is stated in \cite{MSS} for a very ample line bundle, its proof already applies to a based-point-free linear system, since Bertini's theorem provides the smoothness and transversality after pullback to a resolution.} of \cite[Proposition 4.1]{MSS}. Finally, terms for which $q^*s_0$ is identically zero reduce locally to the equation $tz=0$, whose nearby $t$-fibers are contractible, so the claimed vanishing also holds.
\hfill$\diamond$
\end{remark}

\begin{proof}[Proof of Theorem \ref{th-main}]
Let us choose a general hyperplane $H\subset\pp^n$, and set
$X^H:=X\cap H$ and $X^H_{\lambda}:=X_{\lambda}\cap H$.
Let $f^H:=f\vert_H$ denote the restriction of $f$ to
$X^H\setminus X^H_{\infty}$.
Since $H$ is general, the local Euler obstruction of $X^H$
is the restriction of $\Eu_X$ to $X^H$.
Applying Lemma~\ref{baseloc} to $X^H$ therefore yields:
\begin{equation}\label{eq7}
\chi(\Eu_X\vert_{X^H_\infty})
-\chi(\Eu_X\vert_{X^H_0})
=\chi(\varphi_{f^H}(\Eu_{X^H\setminus X^H_\infty})).
\end{equation}
Moreover, the base change isomorphism of
\cite[Lemma 4.3.4]{S}, applied to the general hyperplane $H$,
yields
\begin{equation}\label{eq8}
\varphi_{f^H}(\Eu_{X^H\setminus X^H_\infty})
=\varphi_{f^H}(\Eu_{X\setminus X_\infty}\vert_H)
=\varphi_f(\Eu_{X\setminus X_\infty})\vert_H.
\end{equation}

By additivity of weighted Euler characteristic \eqref{wec}, together with formulae (\ref{eq1}) and (\ref{eq3}), we obtain
\begin{equation}\label{eq5}
\begin{split}
& \UED(X)  - \GED(X) =(-1)^{\dim (X)}	
\left[ \chi(\Eu_{X \setminus (Q \cup H)}) - \chi(\Eu_{X \setminus (Q_\bfw \cup H)})\right] \\
&= (-1)^{\dim (X)}	
\left[ \chi(\Eu_X\vert_{X_\infty \cup X^H}) - \chi(\Eu_X\vert_{X_0 \cup X^H})\right] \\
&= (-1)^{\dim (X)}	
\left[(\chi(\Eu_X\vert_{X_\infty}) - \chi(\Eu_X\vert_{X_0})) - (\chi(\Eu_X\vert_{X^H_\infty}) - \chi(\Eu_X\vert_{X^H_0}))\right].
\end{split}	
\end{equation}
For the second equality, we use the identity 
$\chi(\Eu_{X\setminus Z})=\chi(\Eu_X)-\chi(\Eu_X|_Z)$
for a closed subset $Z\subset X$, so the terms $\chi(\Eu_X)$ cancel.
For the third equality, we get by inclusion-exclusion 
\[
\chi(\Eu_X\vert_{X_\lambda\cup X^H})
=\chi(\Eu_X\vert_{X_\lambda})
+\chi(\Eu_X\vert_{X^H})
-\chi(\Eu_X\vert_{X^H_\lambda}),
\]
and the terms $\chi(\Eu_X\vert_{X^H})$ cancel.

Substituting \eqref{eq6}, \eqref{eq7}, and \eqref{eq8}
into \eqref{eq5}, we get 
\begin{equation*}
\begin{split}	
 \UED(X) - \GED(X) &=
(-1)^{\dim (X)} \left[ 
\chi(\varphi_f (\Eu_{X \setminus X_{\infty}}))-
\chi(\varphi_{f} (\Eu_{X \setminus X_{\infty}})|_H) \right] \\
&= (-1)^{\dim (X)} \chi\left(\varphi_{f} (\Eu_{X \setminus X_{\infty}})|_{X_0\setminus(X_{\infty} \cup X^H)}\right),
\end{split}
\end{equation*}
where the last equality follows from additivity, since
$\varphi_f(\Eu_{X\setminus X_\infty})$ is a constructible
function on $X_0\setminus X_\infty$.
\end{proof}

\begin{proof}[Proof of Theorem \ref{thm-i1}]
This is similar to the proof given in the smooth case  in \cite{ED-degree-MR4136171}; it also follows by combining \cite[equation (3.2)]{ED-degree-MR4623843} with \cite[Theorem 3.2]{ED-degree-MR4623843}. 

Recall that the constructible function
\[ \mu:=\varphi_{f}(\Eu_{X \setminus X_\infty})|_{X_0\setminus(X_\infty \cup H)}
\]
of \eqref{mu} is supported on  $(X_0 \cap \sing_{\mathscr{X}}(f))\setminus (X_\infty \cup H),$ for $\mathscr{X}$ the Whitney stratification of $X$  as before, and we denoted by $\mathscr{X}_0$ the collection of (refined) Whitney strata of $\mathscr{X}$ which are contained in
$$Z:=\overline{X_0 \cap \sing_{\mathscr{X}}(f)}.$$
Then the proof amounts to expressing the constructible function $\mu$ 
 in terms of the basis 
\[
B_1:=\{ {\rm Eu}_{{\overline{V}} \setminus (X_\infty \cup H)} \mid V \in \mathscr{X}_0 \}
\]
of constructible functions with support on $Z \setminus (X_\infty \cup H)$. We include the details of this calculation below.

For each $V \in \mathscr{X}_0$, the corresponding integer $\alpha_V$ appearing in formula \eqref{i3} is defined uniquely by the expression of $\mu$ in the basis $B_1$, that is,
\begin{equation}\label{e1}
    \mu=\sum_{V \in \mathscr{X}_0} \alpha_V \cdot \Eu_{\overline{V}\setminus (X_\infty \cup H)}.
\end{equation}
Furthermore, by applying Theorem \ref{thm2} to the irreducible variety $\overline{V}$ and a general weight $\bfw$ (which, like $H$, was already chosen general with respect to all strata in $\mathscr{X}_0$), we get that
\begin{equation}\label{e2}
    \GED(\overline{V})= (-1)^{\dim (V)} \chi( \Eu_{\overline{V}\setminus (X_\infty \cup H)}).
\end{equation}
 
By combining Theorem \ref{th-main} with \eqref{e1} and \eqref{e2}, we get that
\begin{equation}\begin{split}
\DED(X)
	&= 	(-1)^{\dim (X)-1} \chi\big(\mu\big) \\
    &= \sum_{V \in \mathscr{X}_0} (-1)^{\dim(X)-\dim(V)-1}  \alpha_V \cdot \GED({\overline{V}}).
    \end{split}
	\end{equation}

In order to find the desired formula for $\alpha_V$, let us also express the constructible function $\mu$ in terms of the standard basis 
\[
B_2:=\{ {1}_{V \setminus (X_\infty \cup H)} \mid V \in \mathscr{X}_0 \}
\]
of constructible functions with support on $Z \setminus (X_\infty \cup H)$. In other words, we have
\begin{equation}\label{e3}
\mu=\sum_{V \in \mathscr{X}_0} \mu_V \cdot {1}_{V \setminus (X_\infty \cup H)},
\end{equation}
with $\mu_V=\mu(V)$ for each $V \in \mathscr{X}_0$. 
Evaluating \eqref{e1} and \eqref{e3} at a point $w\in W \setminus (X_\infty \cup H)$, for $W \in \mathscr{X}_0$, yields the equality
\begin{equation}\label{e4}
    \mu_W=\sum_{\{V \mid W \subset {\overline V}\}} \alpha_V \cdot \Eu_{\overline{V}\setminus (X_\infty \cup H)}=\sum_{\{V \mid W \subset {\overline V}\}} \alpha_V \cdot a_{W,V},
\end{equation}
where the integers $a_{W,V}$ are defined by the identity
\begin{equation}
\Eu_{\overline{V}\setminus (X_\infty \cup H)}=\sum_{W \subset {\overline V}} a_{W,V}\cdot 1_{W \setminus (X_\infty \cup H)}. 
\end{equation}
Note that the matrix $A=(a_{W,V})_{W, V\in \mathscr{X}_0}$ is a transition matrix between the bases $B_1$ and $B_2$ mentioned above, and it is an upper-triangular matrix with all diagonal entries equal to $1$. A standard inversion formula applied to \eqref{e4} then yields that
\begin{equation}\label{e5}
    \alpha_W=\sum_{\{V \mid W \subset {\overline V}\}} b_{W,V}\cdot \mu_V,
\end{equation}
with $B=(b_{W,V})_{W, V\in \mathscr{X}_0}$ the transition matrix from $B_2$ to $B_1$, i.e., the inverse of $A$. Then $B$ is also upper triangular with all diagonal entries equal to $1$, and a well-known result of Kashiwara computes the non-zero off-diagonal entries of $B$ in terms of the topology of complex links of pairs of strata. Specifically, for a pair of distinct strata $(W,V)$ in $\mathscr{X}_0$ with $W\subset \overline{V}$ one has that 
\[
b_{W,V}=-\chi_c(L_{W,V}),
\]
with $\chi_c$ the compactly supported Euler characteristic and $L_{W,V}$ the complex link of the pair of strata $(W,V)$. Altogether, for $W \in \mathscr{X}_0$, we get 
\begin{equation}\label{e5b}
    \alpha_W=\sum_{\{V \mid W \subset {\overline V}\}} b_{W,V}\cdot \mu_V = \mu_W - \sum_{\{V \mid W \subset {\overline V} \setminus V\}} \chi_c(L_{W,V})\cdot \mu_V,
\end{equation}
which completes the proof of the theorem.
\end{proof}

\subsection{Computational aspects of our main results}\label{sec:comp}

As we explain next,
the integers $\mu_V:=\mu(V)$ appearing in the statement of Theorem \ref{thm-i1} can be computed in terms of the values of $\Eu_X$ at points in strata of $X$ containing $V$ in their closure. Let us start by noting that we have by definition:
\[
\varphi_{f}(\Eu_{X \setminus X_\infty})=\psi_{f}(\Eu_{X \setminus X_\infty})-(\Eu_{X \setminus X_\infty})\vert_{X_0 \setminus X_\infty}.
\]
Restricting further to $X\setminus(X_\infty \cup H)$ and evaluating at $v\in V\setminus(X_\infty \cup H)$, for $V \in \mathscr{X}_0$, yields:
\begin{equation}\label{def2}
\mu_V=\psi_{f}(\Eu_{X \setminus X_\infty})(v)-\Eu_{X \setminus X_\infty}(v),\end{equation}
with $\psi_{f}:CF(X\setminus X_\infty) \to CF(X_0 \setminus X_\infty)$ the corresponding nearby cycle operator for the function $f$. Both terms on the right-hand side of \eqref{def2} can be expressed as a weighted sum over strata $S\in \mathscr{X}$ different from $V$ and containing $V$ in their closure. Indeed, for a positive dimensional stratum $V \in \mathscr{X}_0$,  let $N_V$ be a general linear subspace of codimension $\dim V$, which meets $V$ transversally at the point $v\in V\setminus(X_\infty \cup H)$, and set $X_V:=X \cap N_V$ and $f_V:=f\vert_{N_V}$. Then, as in  \cite[Section 3]{BLS},
\begin{equation}\label{def4} \Eu_{X\setminus X_\infty}(v)=\Eu_{X_V\setminus X_\infty}(v),\end{equation}
and, similarly, by \cite[Lemma 4.3.4]{S},
\begin{equation}\label{def5} \psi_{f}(\Eu_{X\setminus X_\infty})(v)=\psi_{f_V}(\Eu_{X_V\setminus X_\infty})(v).\end{equation}
Next, by applying \cite[Theorem 3.1]{BLS} (see also \cite{S}) to the germ $(X_V,v)$ with the induced stratification (in which $v$ is a zero-dimensional stratum), and using \eqref{def2}, \eqref{def4} and \eqref{def5}, we can write\footnote{Note that the positive dimensional strata in a stratification of $(X_V,v)$ are of the form $S \cap N_V$, with $S$ strata of $X$, and hence in one-to-one correspondence with the strata $S$ of $X$ different from $V$, and containing $V$ in their closure.}:
\begin{equation}\label{def3}
\mu_V=\sum_{S} 
\left[    
 \chi(f^{-1}({\lambda}) \cap B_{\epsilon}(v) \cap N_V\cap S) 
 - \chi(l^{-1}(\delta) \cap B_{\epsilon'}(v) \cap N_V\cap S) 
\right] \cdot \Eu_X(s),
\end{equation}
where the sum is over strata $S$ of $X$ with $V \subset \overline{S}\setminus S$ and $s\in S$ is any point in the stratum. Moreover, $\epsilon, \epsilon'>0$ are sufficiently small so that $0<\vert {\lambda} \vert \ll \epsilon$, $0<\vert \delta \vert  \ll \epsilon'$, and $B_\epsilon(v)$, $B_{\epsilon'}(v)$ are balls of radius $\epsilon$, resp., $\epsilon'$ centered at $v$ in the ambient space $\bP^n$. Finally, $l$ is a locally defined (near $v$) general linear function.

The expression \eqref{def3} shows that the negative of the integer $\mu_V$ matches what, in \cite[Section 5]{BMPS}, was called the ``defect'' of the function $f$ on $X$ at some point of $V$.

\medskip

\subsection{Isolated singularities}
    Let us now consider the special situation when the projective variety $X$ of dimension $k\geq 1$ has only isolated singularities $p_1,\ldots, p_{s+r}$, with $s,r\in \mathbb{Z}$ satisfying $s\geq 1$ and $r \geq 0$. Assume, moreover, that the stratified critical points of $f$ on its zero fiber are the points $p_1,\ldots, p_{s}$, i.e., $$f^{-1}(0) \cap \sing_{\mathscr{X}}(f)=\{p_1,\cdots,p_s\}.$$ A Whitney stratification of $X$ can be given with a dense $k$-dimensional stratum $S_k=X\setminus \{p_1,\ldots, p_{s+r}\}$ (on which $\Eu_X$ takes the value $1$) and zero-dimensional strata $S_0^i=\{p_i\}$, $i=1,\ldots, s+r$. Moreover, our assumptions on the singularities of $f$ imply that $Z\subseteq \{p_1,\ldots, p_{s}\}$. It then follows from Corollary \ref{cor28} and formula \eqref{def3} that, in the notation above, we have 
    \begin{equation}\label{def_iso}
        \DED(X)
=\sum_{i=1}^s (-1)^{k-1} \cdot \left[ 
 \chi(f^{-1}({\lambda}) \cap B_{\epsilon}(p_i)) 
 - \chi(l^{-1}(\delta) \cap B_{\epsilon'}(p_i)) 
\right].
    \end{equation}
    Let us also note that $F_{f,p_i}:=f^{-1}({\lambda}) \cap B_{\epsilon}(p_i)$ is just the Milnor fiber of $f$ at $p_i$ (and $f$ can be replaced here by any local representative of $f_0$ near $p_i$), while $l^{-1}(\delta) \cap B_{\epsilon'}(p_i)$ coincides with the complex link of $p_i$ in $X$.

\section{Examples}\label{ex-sec}
In this section, we apply the main results  \Cref{th-main} and \Cref{thm-i1} to a pair of concrete examples.

\begin{example}\label{example nodal surface}
    Let $X$ be the complex projective surface 
    \[
    X=\left\{[x_0:x_1:x_2:x_3]\in \PP^3: x_1x_2+\sqrt{-1}x_0x_2=x_3^2\right\}. 
    \]
The projective variety $X$ has an isolated singularity at 
$$p=[1: -\sqrt{-1}:0:0],$$ 
which lies
on the isotropic quadric. 
The intersection of $X$ and the isotropic quadric $Q$ is transversal everywhere except at $p$. 
Thus, we will compute $\DED(X)$ using formula \eqref{def_iso}.

First, we dehomogenize and consider the affine chart $x_0\neq 0$. The affine equation of $X$ is
\[
(z_1+\sqrt{-1})z_2=z_3^2. 
\]
On this affine chart, 
\[
f=\frac{1+z_1^2+z_2^2+z_3^2}{a_0+a_1z_1^2+a_2z_2^2+a_3z_3^2}
\]
for some general $a_i\in \C$ with $0\leq i\leq 3$. 

Define the map $\theta: \C^2\to X$ by
\[
\theta: \C^2\to X, \quad (u, v)\mapsto [1:u^2-\sqrt{-1}:v^2:uv].
\]
This is a double cover of the above affine open set of $X$, which is only ramified at the singular point 
$p$. 
Thus, to compute the Euler characteristics of the local analytic sets $f^{-1}({\bs}) \cap B_{\epsilon}(p)$ and $l^{-1}(\delta) \cap B_{\epsilon'}(p)$ as in formula \eqref{def_iso}, we can compute the Euler characteristic of their preimages in $\C^2$ via $\theta$, and then divide by $2$. Note that none of these two sets contains the singular point $p$, and hence their preimages in $\C^2$ are unramified double covers. 

Since $\theta^{-1}(B_\epsilon(p))$ is a small regular neighborhood of the origin of $\C^2$, we see that $\theta^{-1}(f^{-1}({\bs}) \cap B_{\epsilon}(p))$ is equal to the local Milnor fiber of $f\circ \theta$ at the origin. Since
\begin{align*}
f\circ \theta(u,v)&=\frac{1+(u^2-\sqrt{-1})^2+(v^2)^2+(uv)^2}{a_0+a_1(u^2-\sqrt{-1})^2+a_2(v^2)^2+a_3(uv)^2}\\
&=\frac{u^4-2\sqrt{-1}u^2+v^4+u^2v^2}{a_0+a_1(u^2-\sqrt{-1})^2+a_2(v^2)^2+a_3(uv)^2},
\end{align*}
and
\[
\frac{1}{a_0+a_1(u^2-\sqrt{-1})^2+a_2(v^2)^2+a_3(uv)^2}
\]
is an invertible function near $p$, the Milnor number of $f\circ \theta$ at the origin can be computed algebraically and is equal to $3$. Therefore, the Milnor fiber of $f\circ \theta$ at the origin is homotopy equivalent to the wedge of 3 circles. Therefore,
\[
\chi\big(\theta^{-1}\big(f^{-1}({\bs}) \cap B_{\epsilon}(p)\big)\big)=-2,
\]
and
\[
\chi\big(f^{-1}({\bs}) \cap B_{\epsilon}(p)\big)=-1.
\]
Similarly, we can compute 
\[
\chi\big(l^{-1}(\delta) \cap B_{\epsilon'}(p)\big)=0.
\]
Therefore, by \eqref{def_iso}, 
\[
\DED(X)=1. 
\]

For this example we find  
\[\GED(X)=4,\quad \UED(X)=3.\]
One way to compute these numbers using symbolic computation over $\mathbb{Q}$ is to consider the product of the conjugates 
\[(x_3^2-x_1x_2)^2+x_0^2x_2^2=0.\]
The GED and UED of this new hypersurface are $8$ and $6$, respectively. Since the ED degree is additive over irreducible components, the result follows.
\end{example}

\begin{example}\label{ex:more-complicated-whitney-umbrella}
    
    Let us consider the surface defined by
\[
X=\left\{
[x_0:x_1:x_2:x_3]\in \PP^3:
(x_1+\sqrt{-1}x_0)^2x_0=x_2^2x_3\right\}.
\]
The singular locus of $X$ is $x_1+\sqrt{-1}x_0=x_2=0$. Near $p_1\coloneqq [1:-\sqrt{-1}:0:0]$ and $p_2\coloneqq [0:0:0:1]$, $X$ has the singularity of a Whitney umbrella. 
A Whitney stratification of $X$ is therefore given by
\[
X=X_\textrm{reg}\sqcup \{x_1+\sqrt{-1}x_0=x_2=0, x_0\neq 0, x_3\neq 0\}\sqcup \{p_1\}\sqcup \{p_2\}.
\]
All strata intersect the isotropic quadric $Q$ transversally, except $p_1$. 
Therefore, the support of the constructible function $\varphi_{f} (\Eu_{X \setminus Q_\bfw})$ in Theorem \ref{th-main} is either $p_1$ or empty. Hence, by Theorem \ref{th-main}, we have
\begin{align*}
    \DED(X)&=-\varphi_{f} (\Eu_{X \setminus Q_\bfw})(p_1)\\
    &=-\chi(\Eu_{X}|_{f^{-1}({\bs}) \cap B_{\epsilon}(p_1)})+\Eu_X(p_1).
\end{align*}
Denote by $X^{\mathrm{aff}}$ the subset of $X$ where $x_0\neq 0$. Then, $X^{\mathrm{aff}}$ is the affine variety defined by $(z_1+\sqrt{-1})^2=z_2^2z_3$.
Similar to Example \ref{example nodal surface}, we define a map 
\[
\xi: \C^2\to X^{\mathrm{aff}},\quad (u, v)\mapsto (uv-\sqrt{-1}, u, v^2).
\]
Since, as is well-known,
e.g., see \cite[Example 4.3]{RW},
the value of $\Eu_{X^{\mathrm{aff}}}$ is $1$ along the smooth locus of $X^{\mathrm{aff}}$, it is $2$ along $\{z_1+\sqrt{-1}=z_2=0, z_3\neq 0\}$, and is $1$ at $(-\sqrt{-1}, 0, 0)$, we get
\[
\Eu_{X^{\mathrm{aff}}}=\xi_*(1_{\C^2}).
\]
Therefore, 
\[
\chi(\Eu_{X}|_{f^{-1}({\bs}) \cap B_{\epsilon}(p_1)})=\chi\big(\xi^{-1}\big(f^{-1}({\bs}) \cap B_{\epsilon}(p_1)\big)\big).
\]
Similar to the argument in Example \ref{example nodal surface}, $\xi^{-1}\big(f^{-1}({\bs}) \cap B_{\epsilon}(p_1)\big)$ is homeomorphic to the local Milnor fiber of $f\circ \xi: \C^2\to \C$ at the origin. 
Near the point 
$p_1$,  
up to multiplication by a unit, 
$f$ is represented in the affine chart $x_0\neq 0$ by 
$1+z_1^2+z_2^2+z_3^2.$
 Hence, $f\circ \xi$ is a unit multiple of
\[
 1+(uv-\sqrt{-1})^2+u^2+v^4=u^2v^2-2\sqrt{-1}uv+u^2+v^4.
\]
By algebraic computation, we know that the Milnor number of $f\circ \xi$ at the origin is equal to $1$. Thus, the local Milnor fiber of $f\circ \xi$ at the origin is homotopy equivalent to a circle. Hence,
\[
\chi(\Eu_{X}|_{f^{-1}({\bs}) \cap B_{\epsilon}(p_1)})=\chi\big(\xi^{-1}\big(f^{-1}({\bs}) \cap B_{\epsilon}(p_1)\big)\big)=0.
\]
Therefore, 
\[
\DED(X)=\Eu_X(p_1)=1.
\]

We confirm this calculation using symbolic computation by finding the generic and unit ED degree of the reducible variety defined by 
\begin{align*}
g&=
x_0^6+2 x_0^4 x_1^2+x_0^2 x_1^4+2 x_0^3 x_2^2 x_3-2 x_0x_1^2 x_2^2 x_3+x_2^4 x_3^2    \\
&=\prod_{j=1}^2
\left(
\left(x_1+(-1)^j\sqrt{-1}x_0)^2x_0-x_2^2x_3\right)
\right).
\end{align*}
The polynomial $g$ is the product of complex conjugates, and 
we find that the generic and unit ED degrees are  $10+10$ and $9+9$, respectively. 
\end{example}

\section{Computational challenges motivated by symbolic and numerical methods }\label{s:aag}

In this section we  employ symbolic and numerical methods for additional calculations that motivate conjectures on Euclidean distance degrees and their defects in applications.  While we describe in some detail the basic geometry of the examples mentioned below, key aspects of the computation via our main results appear to be too involved to be carried out by hand. Instead,
all computations here are based on a combination of symbolic computation using \texttt{Macaulay2}~\cite{M2} and numerical homotopy continuation~\cite{bertinibook}. The methods rely on generality assumptions for randomly chosen parameters and on predictor-corrector path tracking, which can occasionally fail due to path crossing or truncation. All reported values were verified across multiple runs.
The code to reproduce all examples is available at this GitHub repository for the \texttt{Macaulay2}~\cite{M2} package \texttt{EuclideanDistanceDegree}~\cite{HR-ED-degree}
\begin{quote}
\url{https://github.com/JoseMath/EuclideanDistanceDegree}\,.
\end{quote}

\newcommand{\hadamard}{\star}

    \subsection{Structured low rank matrices}

We consider examples of $m \times n$ structured low rank matrices. Define 
$X_{m,n,r} \subset \PP^{mn-1}$ to be the projective variety of matrices of rank at most $r$:
    \[
    X_{m,n,r}:=
    \left\{
    \begin{bmatrix}
        x_{11}&\dots&x_{1n}\\
        \vdots&     &\vdots\\
        x_{m1}&\dots&x_{mn}
        \end{bmatrix}\in \PP^{mn-1}:
            \rank\left([x_{i,j}]_{m\times n}\right)\leq r  
    \right\}.
    \]

    \newcommand{\linear}{\mathcal{L}}
    
    Let $\linear \subset \PP^{mn-1}$ be a coordinate linear subspace, i.e., a linear space defined by setting a subset of the coordinates equal to zero. Then the intersection $\linear \cap X_{m,n,r}$ consists of matrices of rank at most $r$
    with a prescribed sparsity pattern given by $\linear$. 
    In general, such intersections need not be irreducible.
    
    When there is no sparsity constraint, 
    the Eckart-Young-Mirsky Theorem computes
    the unit ED degree of $X_{m,n,r}$ by the following closed form formula:
    \[\UED(X_{m,n,r})=\binom{\min\{m,n\}}{r}.\] This is presented in \cite{DHOST} along 
    with a connection to singular value decomposition (SVD) 
    and a discussion of $\GED(X_{m,n,r})$.  In this setting, SVD also provides the global optimal solution to the ED optimization problem:
    \[
    \min_{x\in X_{m,n,r}\cap \mathbb{R}^{m\times n}} \Vert x- u\Vert_F^2,\quad u\in \mathbb{R}^{m\times n},
    \]
    where $\Vert x-u\Vert_F:= \sqrt{\sum_{i=1}^m\sum_{j=1}^n (x_{ij}-u_{ij})^2 }$ 
    denotes the Frobenius norm on matrices. 
    However, once the sparsity constraint is introduced,  minimizing the distance function on $\linear\cap X_{m,n,r}$
    is well-known to be difficult~\cite{KST2022-exact-zeros-ed-degree,OSS2014}.
    We report computations of the GED, UED,  and ED degree defect for such examples.

\begin{example}
Let $X_1$ be the subvariety of $\PP^8$ given by the determinant of 
$\begin{bsmallmatrix}
    x_{0}& x_{1} & x_{2}\\
    x_{3}& x_{4} & x_{5}\\
    x_{6}& x_{7} & x_{8}    
\end{bsmallmatrix}$.
Next, let $X_2=X_1\cap V(x_0)$. That is, we are considering rank at most two matrices where the top-left entry is zero. 
Then 
\[(\GED(X_2),\,\UED(X_2),\,\DED(X_2))=(36,\, 8,\, 28).\]

The singular locus of $X_2$, denoted $\sing(X_2)$,
consists of three irreducible
components and is degree seven.
Two of the three components are the irreducible components of
$\sing(X_1)\cap V(x_0)$;
namely, 
\begin{align*}
V(x_6,x_3,x_0,x_5x_7-x_4x_8,x_2x_7-x_1x_8,x_2x_4-x_1x_5),\\
V(x_2,x_1,x_0,x_5x_7-x_4x_8,x_5x_6-x_3x_8,x_4x_6-x_3x_7).
\end{align*}
The third irreducible component in the singular locus of $X_2$ is the linear space
$V(x_6,x_3,x_2,x_1,x_0)$.
Despite knowing this, 
applying the formulas in~\Cref{ss:statement-results} remains a challenge. The main difficulty is 
understanding how the isotropic quadric intersects $X_2$. 
This motivates the study of computing Whitney stratifications of projective varieties intersected with isotropic quadrics, as well as computing the values of the local Euler obstruction along the strata.
\end{example}

\begin{example}
The next example we consider is $X_3=X_2\cap V(x_8)$. That is rank at most two matrices of the form $\begin{bsmallmatrix}
    0 & x_1& x_2\\
    x_3& x_4& x_5\\
    x_6& x_7& 0   
\end{bsmallmatrix}$.
Then 
\[(\GED(X_3),\,\UED(X_3),\,\DED(X_3))=(33, 25, 8).\]

The singular locus of $X_3$ has six components. Each one is a surface in $\PP^8$. The first four are linear spaces parameterized as follows

\[
\left\{
\begin{bsmallmatrix}
    0& a_0& 0\\
    a_1& a_2& 0\\
    0& 0& 0
\end{bsmallmatrix} \in \PP^8: [a_0:a_1:a_2]\in \PP^2
\right\},
\qquad
\left\{
\begin{bsmallmatrix}
    0& 0& 0\\
    a_0& a_1& a_2\\
    0& 0& 0
\end{bsmallmatrix} \in \PP^8: [a_0:a_1:a_2]\in \PP^2
\right\},
\]
\[
\left\{
\begin{bsmallmatrix}
    0& 0& 0\\
    0& a_0& a_1\\
    0& a_2& 0
\end{bsmallmatrix} \in \PP^8: [a_0:a_1:a_2]\in \PP^2
\right\},
\qquad
\left\{
\begin{bsmallmatrix}
    0& a_0& 0\\
    0& a_1& 0\\
    0& a_2& 0
\end{bsmallmatrix} \in \PP^8: [a_0:a_1:a_2]\in \PP^2
\right\}.
\]
The remaining two components are embeddings of $2\times2$ matrices with rank at most one and are parameterized by $\PP^1\times \PP^1$:
\[
\left\{
\begin{bsmallmatrix}
    0& a_0b_0& a_0b_1\\
    0& a_1b_0& a_1b_1\\
    0& 0& 0
\end{bsmallmatrix} \in \PP^8: ([a_0:a_1],[b_0:b_1])\in \PP^1\times \PP^1
\right\}\]

\[\left\{
\begin{bsmallmatrix}
    0& 0&   0\\
    a_0b_0& a_0b_1& 0\\
    a_1b_0& a_1b_1& 0
\end{bsmallmatrix} \in \PP^8: ([a_0:a_1],[b_0:b_1])\in \PP^1\times \PP^1
\right\}.
\]

Each of the last two components intersects the isotropic quadric in a quartic curve with four singular points. At these singular points, the $2\times 2$ block of nonzero entries is, up to projective scaling, one of the following four matrices:
\[ 
\begin{bsmallmatrix} 1&\sqrt{-1}\\ 
\sqrt{-1}&-1 \end{bsmallmatrix}, 
\quad
\begin{bsmallmatrix} 1&-\sqrt{-1}\\ 
\sqrt{-1}&1 \end{bsmallmatrix}, 
\quad 
\begin{bsmallmatrix} 1&\sqrt{-1}\\ 
-\sqrt{-1}&1 
\end{bsmallmatrix}, 
\quad
\begin{bsmallmatrix} 1&-\sqrt{-1}\\ 
-\sqrt{-1}&-1 \end{bsmallmatrix}.
\] 
The main challenge in this example is computing the value of the local Euler obstruction function at these singular points.
\end{example}

These examples illustrate that ED degree defects can arise even under simple sparsity constraints. 
Important potential applications of our defect formula include the duality conjecture of \cite[Equation~(3.5)]{OSS2014},
as well as extending the results of 
\cite{kozhasov2025minimalalgebraiccomplexityrankone} on the minimal algebraic complexity to higher rank matrices.

\subsection{Intersecting toric hypersurfaces and the sphere}

Toric varieties appear across applications~\cite[Chapter 8]{MS2021-invitation} such as phylogenetics, reaction networks, and algebraic statistics. One reason for their prevalence is because many models are given by monomial maps
naturally leading to a toric~variety. 

\newcommand{\sphere}{X_{\operatorname{sphere}}}

In this section we consider examples of ED degrees of toric hypersurfaces and their 
intersections with the unit sphere.
A known result concerns the ED degree of hypersurfaces defined by a sparse polynomial (i.e., with prescribed monomial support and general coefficients)~\cite{ED-degree-MR4509115}: 
in this setting the ED degree is given by the mixed volume of the associated Lagrange multiplier system. 
When an irreducible hypersurface is defined by a
binomial, it is a toric hypersurface.

We provide a table of computations for binomial hypersurfaces and for their intersection with 
$\sphere = V(x_0^2-x_1^2-x_2^2-\cdots-x_n^2)$.
We take $x^\alpha$ to be $x_1^{\alpha_1}\cdots x_n^{\alpha_n}$ for $\alpha\in \mathbb{N}^{n}$ and $|\alpha|$ to be the sum of the entries. 
The ED degree defect of the binomial hypersurface 
$V(x_0^{|\alpha|}-x^\alpha)$ is always zero; this follows from~\cite{ED-degree-MR3789441}. 
However, a general formula for the UED of toric varieties is not available, even though a combinatorial formula for the GED was given in~\cite{ED-degree-MR3789441}. 
Thus, in view of our main results, 
    a computational challenge is to find combinatorial formulas for the ED degree defect of both toric varieties and their intersections with the $\sphere$.

\begin{table}[hbt!]
    \centering
    \begin{tabular}{c|c|c|c}
         $\alpha\in \mathbb{N}^n$ & $V\left(x_0^{|\alpha|}-x^\alpha\right)$ & $\sphere\bigcap V\left(x_0^{|\alpha|}-x^\alpha\right) $ 
         & Method\\ 
         & GED=UED& (GED, UED, Defect) &\\
(1, 2, 1)& 16& (32, 24, 8)&\texttt{symbolic}\\
         (1, 1, 1)& 12& (24, 18, 6)&\texttt{symbolic} \\
(0,0,1,2) & 6& (30,\,12,\,18)&\texttt{symbolic}\\
(0,1,1,2) & 16& (64,\,32,\,32)&\texttt{numerical}\\ (0,0,1,1) & 4& (20,\,8,\,12)&\texttt{numerical}\\ (1,1,1,2) & 40& (120,\,70,\,50)&\texttt{numerical}\\ (0,1,2,2) & 20& (80,\,40,\,40)&\texttt{numerical}\\ 

\end{tabular}
    \caption{Each triple in the third column reports the GED, UED and ED degree defect of the variety. }
    \label{tab:placeholder}
\end{table}

\medskip

{\bf Acknowledgments.}
L. Maxim acknowledges support from the Simons Foundation (grant MPS-TSM-00007095) and from the project ``Singularities and Applications'' - CF 132/31.07.2023 funded by the European Union - NextGenerationEU - through Romania's National Recovery and Resilience Plan. Jose Israel Rodriguez was partially supported by the Alfred P. Sloan Foundation and by the National
Science Foundation Grant No. 2510307.

\noindent
\footnotesize {\bf Authors' addresses:}
\smallskip

\noindent Laurentiu G. Maxim, University of Wisconsin--Madison, USA \hfill {\tt {maxim@math.wisc.edu}}\newline
\url{https://www.math.wisc.edu/~maxim/}

\smallskip
\noindent Jose Israel Rodriguez, University of Wisconsin--Madison, USA \hfill {\tt  jose@math.wisc.edu}\newline
\url{https://sites.google.com/wisc.edu/jose/}

\smallskip
\noindent Botong Wang, University of Wisconsin--Madison, USA \hfill {\tt wang@math.wisc.edu}\newline
\url{http://www.math.wisc.edu/~wang/}

\begin{thebibliography}{10}

\bibitem{AH}
{\sc P.~Aluffi and C.~Harris}, {\em \bibTitleColor{The {E}uclidean distance degree of smooth complex projective varieties}}, Algebra Number Theory 12 (2018), no. 8, 2005--2032.
\newblock \bibClickableStyle{\href{https://doi.org/10.2140/ant.2018.12.2005}{\texttt{DOI}}}.

\bibitem{bertinibook}
{\sc D.~J. Bates, J.~D. Hauenstein, A.~J. Sommese, and C.~W. Wampler}, \bibTitleColor{{\em Numerically solving polynomial systems with {B}ertini}}, 
 Software, Environments, and Tools, 25. Society for Industrial and Applied Mathematics (SIAM), Philadelphia, PA, 2013.
\newblock \bibClickableStyle{\href{https://doi.org/10.1137/1.9781611972702.ch1}{\texttt{DOI}}}.

\bibitem{BLS}
{\sc J.-P. Brasselet, D.~T. L\^{e}, and J.~Seade}, {\em \bibTitleColor{Euler obstruction and indices of vector fields}}, Topology 39 (2000), no. 6, 1193--1208.
\newblock \bibClickableStyle{\href{https://doi.org/10.1016/S0040-9383(99)00009-9}{\texttt{DOI}}}.

\bibitem{BMPS}
{\sc J.-P. Brasselet, D.~Massey, A.~J. Parameswaran, and J.~Seade}, {\em \bibTitleColor{Euler obstruction and defects of functions on singular varieties}}, J. London Math. Soc. (2) 70 (2004), no. 1, 59--76.
\newblock \bibClickableStyle{\href{https://doi.org/10.1112/S0024610704005447}{\texttt{DOI}}}.

\bibitem{ED-degree-MR4509115}
{\sc P.~Breiding, F.~Sottile, and J.~Woodcock}, {\em \bibTitleColor{Euclidean distance degree and mixed volume}}, Found. Comput. Math. 22 (2022), no. 6, 1743--1765.
\newblock \bibClickableStyle{\href{https://doi.org/10.1007/s10208-021-09534-8}{\texttt{DOI}}}.

\bibitem{Dim04}
{\sc A.~Dimca}, \bibTitleColor{{\em Sheaves in topology}}, Universitext, Springer-Verlag, Berlin, 2004.
\newblock \bibClickableStyle{\href{https://doi.org/10.1007/978-3-642-18868-8}{\texttt{DOI}}}.

\bibitem{DHOST}
{\sc J.~Draisma, E.~Horobe\c{t}, G.~Ottaviani, B.~Sturmfels, and R.~R. Thomas}, {\em \bibTitleColor{The {E}uclidean distance degree of an algebraic variety}}, Found. Comput. Math. 16 (2016), no. 1, 99--149.
\newblock \bibClickableStyle{\href{https://doi.org/10.1007/s10208-014-9240-x}{\texttt{DOI}}}.

\bibitem{GM}
{\sc M.~Goresky and R.~MacPherson}, \bibTitleColor{{\em Stratified {M}orse theory}}, 
 Ergebnisse der Mathematik und ihrer Grenzgebiete (3), 14. Springer-Verlag, Berlin, 1988.
\newblock \bibClickableStyle{\href{https://doi.org/10.1007/978-3-642-71714-7}{\texttt{DOI}}}.

\bibitem{M2}
{\sc D.~R. Grayson and M.~E. Stillman}, {\em \bibTitleColor{Macaulay2, a software system for research in algebraic geometry}}.
\newblock \bibClickableStyle{\href{https://macaulay2.com/}{\texttt{URL}}}.

\bibitem{ED-degree-MR3789441}
{\sc M.~Helmer and B.~Sturmfels}, {\em \bibTitleColor{Nearest points on toric varieties}}, Math. Scand. 122 (2018), no. 2, 213--238.
\newblock \bibClickableStyle{\href{https://doi.org/10.7146/math.scand.a-101478}{\texttt{DOI}}}.


\bibitem{HR-ED-degree}
{\sc W.~Huang and J.~I. Rodriguez}, {\em \bibTitleColor{Euclidean Distance Degrees in Macaulay2}},  (2026).
\newblock \href{https://arxiv.org/abs/2607.13277}{\texttt{arXiv:} \texttt{2607.13277}}.


\bibitem{kozhasov2025minimalalgebraiccomplexityrankone}
{\sc K.~Kozhasov, A.~Muniz, Y.~Qi, and L.~Sodomaco}, {\em \bibTitleColor{On the minimal algebraic complexity of the rank-one approximation problem for general inner products}}, Math. Comp.,  (2025).
\newblock \bibClickableStyle{\href{https://doi.org/10.1090/mcom/4176}{\texttt{DOI}}}.

\bibitem{KST2022-exact-zeros-ed-degree}
{\sc K.~Kubjas, L.~Sodomaco, and E.~Tsigaridas}, {\em \bibTitleColor{Exact solutions in low-rank approximation with zeros}}, Linear Algebra Appl. 641 (2022), 67--97.
\newblock \bibClickableStyle{\href{https://doi.org/10.1016/j.laa.2022.01.021}{\texttt{DOI}}}.

\bibitem{LMW25}
{\sc Y.~Liu, L.~G. Maxim, and B.~Wang}, {\em \bibTitleColor{Maximal twisted {B}etti numbers of complex hyperplane arrangement complements}}, 
 Int. Math. Res. Not. IMRN 2026, no. 6, Paper No. rnag050. \bibClickableStyle{\href{https://doi.org/10.1093/imrn/rnag050}{\texttt{DOI}}}.

\bibitem{Max}
{\sc L.~G. Maxim}, \bibTitleColor{{\em Intersection Homology \& Perverse Sheaves with Applications to Singularities}}, Graduate Texts in Mathematics, 281. Springer, Cham, 2019.
\newblock \bibClickableStyle{\href{https://doi.org/10.1007/978-3-030-27644-7}{\texttt{DOI}}}.

\bibitem{ED-degree-MR4136171}
{\sc L.~G. Maxim, J.~I. Rodriguez, and B.~Wang}, {\em \bibTitleColor{Defect of {E}uclidean distance degree}}, Adv. in Appl. Math. 121 (2020), 102101, 22 pp. \bibClickableStyle{\href{https://doi.org/10.1016/j.aam.2020.102101}{\texttt{DOI}}}.

\bibitem{MRWp}
{\sc L.~G. Maxim, J.~I. Rodriguez, and B.~Wang}, {\em \bibTitleColor{Euclidean distance degree of projective varieties}},  Int. Math. Res. Not. IMRN 2021, no. 20, 15788--15802.
\newblock \bibClickableStyle{\href{https://doi.org/10.1093/imrn/rnz266}{\texttt{DOI}}}.

\bibitem{MRWh}
{\sc L.~G. Maxim, J.~I. Rodriguez, and B.~Wang}, {\em \bibTitleColor{Applications of singularity theory in applied algebraic geometry and algebraic statistics}}, in Handbook of geometry and topology of singularities {VII}, Springer, Cham, 2025, pp.~767--818.
\newblock \bibClickableStyle{\href{https://doi.org/10.1007/978-3-031-68711-2\_14}{\texttt{DOI}}}.

\bibitem{MSS} 
{\sc L.~G. Maxim, M. Saito, and J.~Sch\"{u}rmann}, {\em 
\bibTitleColor{Hirzebruch-Milnor classes of complete intersections}}
Adv. Math. 241 (2013), 220-245.
\newblock
\bibClickableStyle{\href{https://doi.org/10.1016/j.aim.2013.04.001}{\texttt{DOI}}}.

\bibitem{MS22}
{\sc L.~G. Maxim and J.~Sch\"{u}rmann}, {\em \bibTitleColor{Constructible sheaf complexes in complex geometry and applications}}, in Handbook of geometry and topology of singularities {III}, Springer, Cham, 2022, pp.~679--791.
\newblock \bibClickableStyle{\href{https://doi.org/10.1007/978-3-030-95760-5_10}{\texttt{DOI}}}.

\bibitem{ED-degree-MR4623843}
{\sc L.~G. Maxim and M.~Tib\u{a}r}, {\em \bibTitleColor{Euclidean distance degree and limit points in a {M}orsification}}, Adv. in Appl. Math. 152 (2024), Paper No. 102597, 20 pp. \bibClickableStyle{\href{https://doi.org/10.1016/j.aam.2023.102597}{\texttt{DOI}}}.

\bibitem{MS2021-invitation}
{\sc M.~Micha{\l}ek and B.~Sturmfels}, \bibTitleColor{{\em Invitation to nonlinear algebra}}, 
 Graduate Studies in Mathematics, 211. American Mathematical Society, Providence, RI, 2021.
\newblock \bibClickableStyle{\href{https://bookstore.ams.org/gsm-211}{\texttt{URL}}}.

\bibitem{OSS2014}
{\sc G.~Ottaviani, P.-J. Spaenlehauer, and B.~Sturmfels}, {\em \bibTitleColor{Exact solutions in structured low-rank approximation}}, SIAM J. Matrix Anal. Appl. 35 (2014), no. 4, 1521--1542.
\newblock \bibClickableStyle{\href{https://doi.org/10.1137/13094520X}{\texttt{DOI}}}.


\bibitem{RW} {\sc J. I. Rodriguez and B. Wang}, {\em \bibTitleColor{Computing Euler obstruction functions using maximum likelihood degrees}},
Int. Math. Res. Not. IMRN 2020, no. 20, 6699--6712.
\bibClickableStyle{\href{https://doi.org/10.1093/imrn/rnz243}{\texttt{DOI}}}.

\bibitem{S}
{\sc J.~Sch\"{u}rmann}, \bibTitleColor{{\em Topology of singular spaces and constructible sheaves}}, Monografie Matematyczne 63, Birkh\"{a}user Verlag, Basel, 2003.
\newblock \bibClickableStyle{\href{https://doi.org/10.1007/978-3-0348-8061-9}{\texttt{DOI}}}.

\end{thebibliography}
\end{document}